\documentclass[11pt,letterpaper]{amsart}

\usepackage{amsmath,amssymb,amsfonts,amsthm,mathrsfs}
\usepackage{graphicx}
\usepackage[usenames,dvipsnames]{color}

\usepackage[margin=1in]{geometry}

\theoremstyle{plain}
\newtheorem{thm}{Theorem}
\newtheorem{lem}[thm]{Lemma}
\newtheorem{cor}[thm]{Corollary}

\theoremstyle{definition}
\newtheorem{remark}[thm]{Remark}

\newcommand{\nat}{\ensuremath {\mathbb N} }

\newcommand{\remove}[1] {}

\newcommand{\ex} {{\bf E}}
\newcommand{\pr} {{\bf Pr}}

\newcommand{\sG} {\ensuremath{\mathscr G}}

\newcommand{\eps}{\varepsilon}

\DeclareMathOperator{\Bin}{Bin}

\title{On-line list colouring of random graphs}
\date{}

\author{Alan Frieze}
\address{Department of Mathematical Sciences, Carnegie Mellon University, 5000 Forbes Av., 15213, Pittsburgh, PA, U.S.A}
\thanks{The first author is supported in part by NSF Grant CCF0502793}
\email{alan@random.math.cmu.edu}

\author{Dieter Mitsche}
\address{Universit\'{e} de Nice Sophia-Antipolis, Laboratoire J-A Dieudonn\'{e}, Parc Valrose, 06108 Nice cedex 02}
\email{\texttt{dmitsche@unice.fr}}

\author{Xavier P\'{e}rez-Gim\'{e}nez}
\address{Department of Mathematics, Ryerson University, Toronto, ON, Canada}
\email{\texttt{xperez@ryerson.ca}}

\author{Pawe\l{} Pra\l{}at}
\address{Department of Mathematics, Ryerson University, Toronto, ON, Canada}
\thanks{The fourth author is supported in part by NSERC}
\email{\tt pralat@ryerson.ca}

\begin{document}
\maketitle

\begin{abstract}
In this paper, the on-line list colouring of binomial random graphs $\sG(n,p)$ is studied. We show that the on-line choice number of $\sG(n,p)$ is asymptotically almost surely asymptotic to the chromatic number of $\sG(n,p)$, provided that the average degree $d=p(n-1)$ tends to infinity faster than $(\log \log n)^{1/3} (\log n)^2 n^{2/3}$. For sparser graphs, we are slightly less successful; we show that if $d \ge (\log n)^{2+\eps}$ for some $\eps>0$, then the on-line choice number is larger than the chromatic number by at most a multiplicative factor of $C$, where $C \in [2,4]$, depending on the range of $d$. Also, for $d=O(1)$, the on-line choice number is by at most a multiplicative constant factor larger than the chromatic number.
\end{abstract}

\section{Introduction} 

The combinatorial game we study in the paper is played by two players, named Mr.\ Paint and Mrs.\ Correct, and is played on a finite, undirected graph in which each vertex has assigned a non-negative number representing the number of erasers at the particular vertex. We assume for simplicity that this number is initially the same for each vertex. In each round, first Mr.\ Paint selects a subset of the vertices and paints them all the same colour; he cannot use this colour in subsequent rounds. Mrs.\ Correct then has to erase the colour from some of the vertices in order to prevent adjacent vertices having the same colour. Whenever the colour at a vertex is erased, the number of erasers at that vertex decreases by $1$, but naturally, Mrs.\ Correct cannot erase the colour if she has no erasers left at that vertex. Vertices whose colours have not been erased can be considered as being permanently coloured and can be removed from the game. The game has two possible endings: (i) all vertices have been permanently coloured, in which case Mrs.\ Correct wins, or (ii) at some point of the game, Mr.\ Paint presents two adjacent vertices $u$ and $v$ and neither $u$ nor $v$ has any eraser left, in which case Mr.\ Paint wins. If, regardless of which sets she gets presented, there is a strategy for Mrs.\ Correct to win the game having initially $k-1$ erasers at each vertex, we say that the graph is \textbf{$k$-paintable}. The smallest $k$ for which the graph is $k$-paintable is called the \textbf{paintability} number of $G$, and denoted by $\chi_P(G)$. Note that this parameter is indeed well defined: for any graph on $n$ vertices, $n-1$ erasers at each vertex always guarantee Mrs.\ Correct to win, as she can always choose one vertex from a set presented to her and erase colours on the remaining ones. This problem is also known as the \textbf{on-line list colouring} and the corresponding graph parameter is also called the \textbf{on-line choice number of $G$}---see, for example,~\cite{Xuding} and below for the relation to the (off-line) list colouring.

\bigskip

Let us recall a classic model of random graphs that we study in this paper. The \textbf{binomial random graph} $\sG(n,p)$ is defined as the probability space $(\Omega, \mathcal{F}, \Pr)$, where $\Omega$ is the set of all graphs with vertex set $\{1,2,\dots,n\}$, $\mathcal{F}$ is the family of all subsets of $\Omega$, and for every $G \in \Omega$,
$$
\Pr (G) = p^{|E(G)|} (1-p)^{{n \choose 2} - |E(G)|} \,.
$$
This space may be viewed as the set of outcomes of ${n \choose 2}$ independent coin flips, one for each pair $(u,v)$ of vertices, where the probability of success (that is, adding edge $uv$) is $p.$ Note that $p=p(n)$ may (and usually does) tend to zero as $n$ tends to infinity.  

All asymptotics throughout are as $n \rightarrow \infty $ (we emphasize that the notations $o(\cdot)$ and $O(\cdot)$ refer to functions of $n$, not necessarily positive, whose growth is bounded; whereas $\Theta(\cdot)$ and $\Omega(\cdot)$ always refer to positive functions). We say that an event in a probability space holds \textbf{asymptotically almost surely} (or \textbf{a.a.s.}) if the probability that it holds tends to $1$ as $n$ goes to infinity. We often write $\sG(n,p)$ when we mean a graph drawn from the distribution $\sG(n,p)$.  Finally, for simplicity, we will write $f(n) \sim g(n)$ if $f(n)/g(n) \to 1$ as $n \to \infty$ (that is, when $f(n) = (1+o(1)) g(n)$). 

\bigskip

Now, we will briefly mention the relation to other known graph parameters. A \textbf{proper colouring} of a graph is a labelling of its vertices with colours such that no two adjacent vertices have the same colour. A colouring using at most $k$ colours is called a (proper) \textbf{$k$-colouring}. The smallest number of colours needed to colour a graph $G$ is called its \textbf{chromatic number}, and it is denoted by $\chi(G)$. Let $L_k$ be an arbitrary function that assigns to each vertex of $G$ a list of $k$ colours. We say that $G$ is \textbf{$L_k$-list-colourable} if there exists a proper colouring of the vertices such that every vertex is coloured with a colour from its own list. A graph is \textbf{$k$-choosable}, if for every such function $L_k$, $G$ is $L_k$-list-colourable. The minimum $k$ for which a graph is $k$-choosable is called the \textbf{list chromatic number}, or the \textbf{choice number}, and denoted by $\chi_L(G)$. Since the choices for $L_k$ contain the special case where each vertex is assigned the list of colours $\{1, 2, \ldots,k\}$, it is clear that a $k$-choosable graph has also a $k$-colouring, and so $\chi(G) \leq \chi_L(G)$. It is also known that if a graph is $k$-paintable, then it is also $k$-choosable~\cite{Schauz}, that is, $\chi_L(G) \leq \chi_P(G)$. Indeed, if there exists a function $L_k$ so that $G$ is not $L_k$-list-colourable, then Mr.\ Paint can easily win by fixing some permutation of all colours present in $L_k$ and presenting at the $i$-th step all vertices containing the $i$-th colour of the permutation on their lists (unless the vertex was already removed before). Finally, it was shown in~\cite{Xuding} that the paintability of a graph $G$ on $n$ vertices is at most $\chi(G) \log n + 1$. (All logarithms in this paper are natural logarithms.) Combining all inequalities we get the following:
\begin{equation}\label{eq:chi}
\chi(G) \leq \chi_L(G) \le \chi_P(G) \le \chi(G) \log n + 1.
\end{equation}

\bigskip

It follows from the well-known results of Bollob\'as~\cite{Bol88}, \L uczak~\cite{Luc91} (see also McDiarmid~\cite{McD}) that the chromatic number of $\sG(n,p)$ a.a.s.\ satisfies
\begin{equation}\label{eq:chi2}
\chi(\sG(n,p)) \sim \frac {\log(1/(1-p))n}{2 \log (np)},
\end{equation}
for $np\to\infty$ and $p$ bounded away from $1$. The study of the choice number of $\sG(n,p)$ was initiated in~\cite{Alon92}, where Alon proved that a.a.s., the choice number of $\sG(n, 1/2)$ is $o(n)$. Kahn then showed (see~\cite{AlonS}) that a.a.s.\ the choice number of $\sG(n,1/2)$ equals $(1+o(1))\chi_{\sG(n,1/2)}$. In~\cite{Krivelevich}, Krivelevich showed that this holds for $p \gg n^{-1/4}$, and Krivelevich, Sudakov, Vu, and Wormald~\cite{KSVW02} improved this to $p \gg n^{-1/3}$. On the other hand, Alon, Krivelevich, Sudakov~\cite{AKS97} and Vu~\cite{Vu} showed that for any value of $p$ satisfying $2 < np \leq n/2$, the choice number is $\Theta(np/\log(np))$. Later, Krivelevich and Vu~\cite{KV02a} generalized this to hypergraphs; they also improved the leading constants and showed that the choice number for $C \leq np \leq 0.9n$ (where $C$ is a sufficiently large constant) is at most a multiplicative factor of $2+o(1)$ away from the chromatic number, the best known factor for $p \leq n^{-1/3}$. Our results below (see Theorem~\ref{thm:main}, Theorem~\ref{thm:main2}, and Theorem~\ref{thm:main3}) show that even for the on-line case, for a wide range of $p$, we can asymptotically match the best known constants of the off-line case. Moreover, if $np \ge \log^{\omega} n$ (for some function $\omega=\omega(n)$ tending to infinity as $n \to\infty$), then we get the same multiplicative factor of $2+o(1)$.

\bigskip

Our main results are the following theorems. The first one deals with dense random graphs.

\begin{thm}\label{thm:main}
Let $\eps > 0$ be any constant, and suppose that 
$$
(\log \log n)^{1/3} (\log n)^2 n^{-1/3} \ll p \leq 1-\eps.
$$ 
Let $G \in \mathcal{G}(n,p)$. Then, a.a.s., 
$$
\chi_P(G) \sim \frac {n}{2 \log_b (np)} \sim \chi(G),
$$
where $b = 1/(1-p)$.
\end{thm}

Note that if $p=o(1)$, then 
$$
\frac {n}{2 \log_b (np)} = \frac {n \log (1/(1-p))}{2 \log (np)} \sim \frac {np}{2 \log (np)} = \Theta \left( \frac {np}{\log (np)} \right).
$$
For constant $p$ it is not true that $\log (1/(1-p)) \sim p$ but the order is preserved, provided $p\leq 1-\eps$ for some $\eps>0$.

\bigskip

For sparser graphs we are less successful in determining the asymptotic behaviour of $\chi_P(\mathcal{G}(n,p))$. Nevertheless, we can prove the following two theorems that determine the order of the graph parameter we study.

\begin{thm}\label{thm:main2}
Let $\eps > 0$ be any constant, and suppose that 
$$
\frac {(\log n)^{2+\eps}}{n} \leq p = O((\log \log n)^{1/3} (\log n)^2 n^{-1/3}).
$$

Let $G \in \mathcal{G}(n,p)$. Then, a.a.s., 
$$
\chi_P(G)=\Theta \left( \frac {np}{\log(np)} \right) =\Theta(\chi(G)).
$$
Moreover, if $np = (\log n)^{C+o(1)}$, a.a.s.\
$$
\chi(G) \le \chi_P(G) \le (1+o(1))
\begin{cases}
2  \chi(G) & \text{if } C \to \infty \\
\frac {2C}{C-2} \chi(G) & \text{if } C \in [4, \infty) \\
4  \chi(G) & \text{if } C \in (2,4).
\end{cases}
$$
\end{thm}

\bigskip

Finally, for very sparse graphs we have the following.

\begin{thm}\label{thm:main3}
Let $G \in \mathcal{G}(n,p)$ with $p=O(1/n)$. Then, a.a.s., $\chi_P(G)=\Theta(1)=\Theta(\chi(G))$.
\end{thm}

\section{Preliminaries}

Most of the time, we will use the following version of \textbf{Chernoff's bound}. Suppose that $X \in \Bin(n,p)$ is a binomial random variable with expectation $\mu=np$. If $0<\delta<1$, then 
$$
\pr [X < (1-\delta)\mu] \le \exp \left( -\frac{\delta^2 \mu}{2} \right),
$$ 
and if $\delta > 0$,
\[\pr [ X > (1+\delta)\mu] \le \exp\left(-\frac{\delta^2 \mu}{2+\delta}\right).\]
However, at some point we will need the following, stronger, version: for any $t \ge 0$, we have
\begin{equation}\label{eq:strongChernoff}
\pr [X \ge \mu + t] \le \exp \left( - \mu \varphi \left( \frac {t}{\mu} \right) \right),
\end{equation}
where
$$
\varphi(x) = (1+x)\log(1+x)-x.
$$
These inequalities are well known and can be found, for example, in~\cite{JLR}.

\bigskip

Let $G=(V,E)$ be any graph. A set $S \subseteq V$ is called \textbf{independent} if no edge $e \in E$ has both endpoints in $S$. Denote by $\alpha(G)$ the \textbf{independence number} of $G$, that is, the size of a largest independent set of $G$. Let $k_0$ be defined as follows: 
$$
k_0=k_0(n,p)=\max \left\{k \in \nat: \binom{n}{k}(1-p)^{\binom{k}{2}} \geq n^4 \right\}.
$$
It is well known that $k_0$ is well defined (for all $n$ sufficiently large and provided that $p \le 1 - \eps$ for some $\eps>0$) and that $k_0 \sim 2 \log_b (np)$ with $b=1/(1-p)$. The following result was proved in~\cite{KSVW02}.

\begin{thm}[\cite{KSVW02}]\label{lem:Krivelevich}
Suppose $n^{-2/5} \log^{6/5} n \ll p \leq 1-\varepsilon$ for some constant $\varepsilon > 0$. Let $G \in \mathcal{G}(n,p)$. Then
$$
\pr(\alpha(G) < k_0) = \exp\left(-\Omega\left( \frac{n^2}{k_0^4 \; p}\right)\right).
$$
\end{thm}

We obtain the following immediate corollary that will be useful to deal with dense random graphs. In fact, this lemma is the bottleneck for the argument that prevents us to extend Theorem~\ref{thm:main} for sparser graphs.

\begin{cor}\label{cor:largesets}
Let $\eps>0$ be any constant and let $\omega = \omega(n) = o(\log n)$ be any function tending to infinity as $n \to \infty$. Suppose that
$$
\omega (\log \log n)^{1/3} (\log n)^2 n^{-1/3}  \leq p \leq 1-\varepsilon.
$$
Let $G \in \mathcal{G}(n,p)$. Then, a.a.s.\ every set $S \subseteq V(G)$ with $|S| =s \geq s_0 := n/(\omega \log^2 n)$ contains an independent set of size $k_0 = k_0(s,p) \sim k_0(n,p)$.
\end{cor}
\begin{proof}
Fix a set $S \subseteq V(G)$ with $|S| = s \geq n/(\omega \log^2 n)$. First, let us note that 
$$
k_0 = k_0(s,p) \sim 2 \log_b (sp) = 2 \log_b (np) \left( 1 - O\left( \frac {\log \log n}{\log n} \right) \right) \sim k_0(n,p).
$$
Moreover, since $n/(\omega \log^2n) \le s \le n$, we can easily verify that $p$ satisfies $s^{-2/5} \log^{6/5} s \ll p \leq 1-\varepsilon$ (with room to spare in the lower bound). It follows immediately from Theorem~\ref{lem:Krivelevich} that the probability of not having an independent set of size $k_0$ in $S$ is at most 
$$
\exp\left(- \Omega \left( \frac {s^2}{k_0^4 p} \right) \right) = \exp\left(- \Omega \left( \frac {s^2 p^3}{\log^4 n} \right) \right) = \exp\left(- \Omega \left( \frac {s n p^3}{\omega \log^6 n} \right) \right) = \exp\left(- \Omega \left( s \omega^2 \log \log n \right) \right).
$$
On the other hand, the number of sets of size $s$ can be bounded as follows:
$$
{n \choose s} \le \left( \frac {ne}{s} \right)^s = \exp \left( s \log (ne/s) \right) \le \exp \left( 3 s \log \log n \right),
$$
since $\log (ne/s) \le \log (e \omega \log^2 n) = (2+o(1)) \log \log n$. Hence, the expected number of sets of size $s$ for which the desired property fails is $\exp\left(- \Omega \left( s \omega^2 \log \log n \right) \right)$. Summing over all $s \geq s_0$, the expected number of all such sets is  $\exp\left(- \Omega \left( s_0 \omega^2 \log \log n \right) \right) = o(1)$. The claim holds by Markov's inequality. 
\end{proof}

\bigskip

Given a graph $G$ and some fixed ordering of its vertices (for example, we may assume that it is the natural ordering $1, 2, \ldots,n$), the greedy colouring algorithm proceeds by scanning vertices of $G$ in the given order and assigning the first available colour for the current vertex.  We will use the following lemma due to~\cite{KV02} (see also~\cite{CojaTaraz03}). However, since we use it in a slightly different setting, we point out a few small differences in the argument by providing a sketch of the proof. In fact, this lemma is the bottleneck for the argument that prevents us to extend Theorem~\ref{thm:main2} for sparser graphs.

\begin{lem}\label{lem:largesets2}
Let $\omega = \omega(n) := \log \log n$. Given any constant $0 < \eps < 1$, let $G \in \mathcal{G}(n,p)$ with $\log^{2+\eps} n/n =: p_0 \leq p=o(1/\log n)$. Then, a.a.s.\ every subgraph of $G$ of size $u \geq u_0 := n/(\omega \log^2 n)$ has an independent set of size at least $\eps(1-\eps) \log (np)/(3p)$.
\end{lem}

Note that in the lemma we set $\omega=\log \log n$; other functions tending to infinity slower than $\log \log n$, constant, or even tending to $0$ not too quickly clearly would work as well, but would make the statement of the result weaker.

\begin{proof}[Sketch of the proof] We follow the notation as in~\cite{KV02} and apply the greedy approach given there for each subgraph of size $u$. Let  
$$
\alpha_0=\frac{(1-\eps) \log (up)}{p} \hspace{0.5in} \text{ and  } \hspace{0.5in} t=\frac{up}{\log (up)}. 
$$
Note that
\begin{eqnarray*}
\alpha_0 &\geq& \frac{(1-\eps) \log (np/(\omega \log^2 n))}{p} = \frac{(1-\eps) (\log (np) - (2+o(1)) \log \log n)}{p} \\
&\ge& \frac{(1-\eps) (\eps+o(1)) \log (np)}{(2+\eps) p} \ge \frac{(1-\eps)\eps \log (np)}{3p}, 
\end{eqnarray*}
since $np \ge \log^{2+\eps} n$. (Note that, if $np \ge \log^C n$, then we get the better estimate $$\alpha_0 \ge (1-\eps)(1-2/C+o(1)) \log(np)/p;$$ see Remark~\ref{rem:densecase} below.) Moreover, 
\begin{eqnarray*}
(1-p)^{\alpha_0} &=& \exp \left( -p \alpha_0 (1+O(p)) \right) = \exp \left( - (1-\eps) \log(up) (1+O(p)) \right) \\
&\sim& \exp \left( - (1-\eps) \log(up) \right) = (up)^{-1+\eps},
\end{eqnarray*}
since $p = o(1/\log n)$. For a fixed subgraph of size $u$, it follows from~\cite{KV02} that the probability that the algorithm fails to produce an independent set of size at least $\alpha_0$ is at most 
\begin{align*}
\exp\left(-t(1-p)^{\alpha_0}u\eps\right)=&\exp \left( -\frac{(up) (1+o(1)) (up)^{-1+\eps} u \eps}{\log (up) } \right) \\
=& \exp \left( - \frac{u^{1+\eps} p^{\eps} (\eps+o(1))}{\log (up)} \right).
\end{align*}
By taking a union bound over all 
$$
\binom{n}{u} \leq \left( \frac {ne}{u} \right)^u = \exp(u \log(ne/u) ) \le \exp(3 u \log \log n)
$$  
sets of size $u$, the probability $P_u$ that there exists a subgraph of size $u$ for which the algorithm fails is at most
$$
\exp \left( - u \left( \frac{(up)^{\eps} (\eps+o(1))}{\log (up)} - 3 \log \log n \right) \right) \leq \exp \left( - u \left( \frac{(u_0 p_0)^{\eps} (\eps+o(1))}{\log (u_0p_0)} - 3 \log \log n \right) \right),
$$
as the function $f(x) = x^{\eps}/ \log x$ is increasing for large enough $x$. Since $u_0 p_0 = \log^{\eps} n / \omega = \log^{\eps+o(1)} n$,
\begin{align*}
P_u \le& \exp \left( - u \left( \frac {(\log n)^{(\eps+o(1))\eps} (\eps+o(1))}{(\eps+o(1)) \log \log n} - 3 \log \log n \right) \right) \\
=& \exp \left( - u \left( (\log n)^{\eps^2+o(1)} - 3 \log \log n \right) \right) \le e^{-u}.
\end{align*}
Summing over all $u_0 \leq u \leq n$, we see that the probability that the algorithm fails is at most
$
\sum_{u=u_0}^n e^{-u} = O(e^{-u_0}) = o(1),
$
and the lemma follows.
\end{proof}

\bigskip

\begin{remark}\label{rem:densecase}
Lemma~\ref{lem:largesets2} is sufficient to determine the order of the on-line choice number for sparse random graphs. We comment on improvements in order to obtain the smallest leading constant in the upper bound. From the proof of the lemma it is clear that the bound on the size of the independent set can be improved for denser graphs. More precisely, for $np \geq \log^{\omega} n$ (for $\omega\to\infty$ but still $p = o(1/\log n)$), we can obtain independent sets of size at least $(1-2\eps)\log(np)/p$ for any arbitrarily small $\eps$, and so we are asymptotically just a factor $2+o(1)$ off from the bound of Corollary~\ref{cor:largesets}. 

On the other hand, for $C$ being a constant larger than $2$, and $np=\log^{C+o(1)} n$, we can obtain independent sets of size at least $(1-\eps)(1-2/C+o(1)) \log(np)/p$ (again, for any arbitrarily small $\eps$), so by Lemma~\ref{lem:largesets2} we are off by a factor $2/(1-2/C)+o(1) = 2C/(C-2) + o(1)$. We will now show that for $\log^{2+\eps} n/n \leq p=o(n^{-5/6})$, however, the independence number obtained is by at most a factor $4+o(1)$ off from the one obtained by Corollary~\ref{cor:largesets}. First, as shown above, given $u \ge n/(\omega \log^2 n)$, there are at most $\exp(3u \log \log n)$ sets of size $u$, and the expected number of edges induced by each such set is $\binom{u}{2}p$. By Chernoff's bound, the probability that the number of edges induced by one of them deviates by an additive $\binom{u}{2}p / \log\log n$ factor from its expected value is at most $\exp\left(-(1/\log \log n)^2 u^2 p/5\right)$. Hence, with probability at most 
$$
\exp\left(u (3 \log \log n- (\log n)^{\eps} (1/\log \log n)^2/(5\omega)) \right) \leq \exp\left(-u \right),
$$
there exists a set of size $u$ that does not satisfy the condition. Summing over all $n/(\omega \log^2 n) \leq u \leq n$, we see that a.a.s.\ for all such sets we have $(1+o(1)) u^2 p/2$ edges.  On the other hand, note that for this range of $p$, the expected number of pairs of triangles sharing one vertex in $G$ is $\Theta(n^5 p^6)=o(1)$, and thus by Markov's inequality, a.a.s.\ every vertex is in at most one triangle. It follows that there exists a  set $D$ of edges which is a matching such that after removing $D$ the graph is triangle-free. Since in $G$  the average degree of every set of size $u$ is $up(1+o(1))$, the same also holds for $G \setminus D$. Since $G \setminus D$ is triangle-free, by Shearer's lemma~\cite{Shearer}, an independent set of size $(1+o(1))\log (np)/p$ can be found. By eliminating at most half of the set (one vertex for each edge in $D$), we obtain an independent set in the original graph. It follows that for every set of size $u$, there exists an independent set of size at least $\log (np)/((2+o(1))p)$, being therefore at most a factor $4+o(1)$ off from the size obtained by Corollary~\ref{cor:largesets}.
\end{remark}

\bigskip

We will also need the following observation.

\begin{lem}\label{lem:mediumsets}
Let $G \in \mathcal{G}(n,p)$, for $10\log n/n \le p \le 1$, and set $s_0=10 \log n/p$. Then, a.a.s.\ the following holds:
\begin{itemize}
\item [(i)] every set $S \subseteq V(G)$ with $s_0 \leq |S| \le n$ contains an independent set of size at least $1/(9p)$;
\item [(ii)] for every $i\in\nat$ and every set $S \subseteq V(G)$ with $2^{-i} s_0 \leq |S| < 2^{-i+1} s_0$, $S$ contains an independent set of size at least $i/(9p 2^i)$. 
\end{itemize}
\end{lem}

Before we move to the proof of the lemma, let us note that the lower bound on $p$ is not used in the proof of part~(i), which becomes a trivial statement if $p<10 \log n/n$. For part~(ii), we could easily relax this bound to $p\ge 1/n^k$ (for any constant $k>0$) by changing some of the constants in the statement. Furthermore, part~(i) is trivially true for $p\ge1/9$, and part (ii) is only non-trivial for all $i$ satisfying $9p2^i<i$.

\remove{Before we move to the proof of the lemma, let us note that the lower bound for $p$ is not used in the proof of the result. It is introduced to make sure the statement is non-trivial. Similarly, the condition for $i$ in part (ii) is to make sure that we claim the existence of an independent set of size larger than 1.}

\begin{proof}
For part (i), let us fix a set $S$ with $|S|=s$ satisfying $s_0 \le s\le n$. Denote by $E_S$ the set of edges induced by $S$. We have $\ex{|E_S|}=\binom{s}{2}p \sim s^2 p/2$. By Chernoff's bound, with probability at least $1-\exp(-\ex{|E_S|}/4)$, we have $|E_S| \leq s^2 p$. By taking a union bound over all $\binom{n}{s}\leq \exp(s \log n)$ subsets of size $s$, with probability at most $\exp(s (\log n - s p/9))$ there exists a set of size $s$ not satisfying the condition. Since, by assumption, $s \geq s_0 = 10 \log n/p$, this probability is at most $\exp(- s \log n / 9)$, and so summing over all $s \ge s_0$, the claim holds a.a.s.\ for all $s$ in the desired range. We may therefore assume that (deterministically) every subset of size $s \ge s_0$ satisfies $|E_S| \leq s^2 p$. Then, for any set $S$ in the desired range, at least $s/2$ vertices of $S$ have a degree of at most $4sp$
in the graph induced by $S$.
By applying a greedy algorithm to these vertices, we can therefore always find an independent set of size at least
$(s/2)/(4sp + 1) \ge 1/(9 p)$. 	

Part (ii) is proved in a similar way. Fix $i = i(n) \in \nat$ such that
\begin{equation}\label{eq:i1}
9p 2^i < i
\end{equation}
(since otherwise the statement is trivial), and a set $S\subseteq V(G)$ of size $s$ satisfying the following: 
\begin{equation}\label{eq:i2}
2^{-i} s_0 \le s < 2^{-i+1} s_0.
\end{equation}

Since $s$ is a natural number, we must have $2^{-i+1} s_0 \ge 1$. Therefore, $2^i \le 2 s_0 \le 2n$ (where we used that $p\ge10\log n/n$), and thus
\begin{equation}\label{eq:i3}
i \le \log(2n)/\log 2 \le 2 \log n,
\end{equation}
for $n$ large enough. Combining~\eqref{eq:i2}, \eqref{eq:i1}, the fact that $ps_0 = 10 \log n$, and~\eqref{eq:i3}, we get
\[
s \ge 2^{-i} s_0 > (9p/i) s_0 =  (90/i) \log n \ge 45.
\]
Note that for $s\ge45$,
\[
\ex{|E_S|}=\binom{s}{2}p \ge \frac {s^2 p}{2.1}.
\]
It follows from the stronger version of Chernoff's bound~\eqref{eq:strongChernoff} that 
\begin{eqnarray*}
\pr \left( |E_S| \ge \left( 1+\frac {2^i}{i} \right) \ex{|E_S|} \right) &=& \exp \left( - \varphi(2^i/i) \ex{|E_S|} \right) \\
&\le& \exp \left( - \frac {2^{i}}{4} \cdot \frac {s^2 p}{2.1} \right) \\
&\le& \exp \left( - \frac {2^{i} s^2 p}{9} \right).
\end{eqnarray*}
(Note that $\varphi(2^i/i) = (1+2^i/i)\log(1+2^i/i)-2^i/i \sim (2^i/i) \log(2^i/i) \sim 2^i \log 2$ as $i \to \infty$. Moreover, it is straightforward to see that for every $i \in \nat$, $\varphi(2^i/i) \ge 2^i/4$.)

As before, by a union bound we get that with probability at most $\exp(s (\log n - 2^i s p/9))$ there exists a set of size $s$ not satisfying the condition. Since, by assumption, $s \geq 2^{-i} s_0 = 2^{-i} 10 \log n/p$, this probability is at most $\exp(- s \log n/9)$,
regardless of which precise interval $[2^{-i}s_0,2^{-i+1}s_0)$ contains $s$.
Summing the $\exp(- s \log n/9)$ bound we obtained over all values $45\le s < s_0$ gives $o(1)$.
Therefore, we may assume that (deterministically), for any $i\in \nat$, every subset $S$ of size $s$ in the range~\eqref{eq:i2} satisfies
$$
|E_S| \leq \left(1+\frac {2^i}{i} \right)  \binom{s}{2}p \le \left(1+\frac {2^i}{i} \right) \frac {s^2 p}{2}    \le 2^i s^2 p / i.
$$
Arguing as before, this guarantees that we can always find an independent set of size at least
\[
(s/2)/(4 \cdot 2^i sp/i + 1) \ge i/(9 \cdot 2^i p).
\]
(At the last step, we used that $4 \cdot 2^i sp/i + 1 \le 4.2 \cdot 2^i sp/i$, which follows easily from~\eqref{eq:i2}, \eqref{eq:i3} and the definition of $s_0$.)
The proof of the lemma is finished.
\end{proof}
\section{Proof of Theorem~\ref{thm:main}}

The lower bound follows immediately from~\eqref{eq:chi}. For the upper bound, we give a winning strategy for Mrs.\ Correct that a.a.s.\ requires only $(1+o(1)) n / (2 \log_b (np))$ erasers on each vertex, where $b = 1/(1-p)$. (Recall from~\eqref{eq:chi2} that a.a.s.\ $\chi(G) \sim n / (2 \log_b (np))$.)
We emphasize that most probabilistic statements hereafter in the proof will not refer to $\sG(n,p)$ but rather to the randomized strategy that Mrs.\ Correct uses to select each independent set, regardless of the strategy of Mr.\ Paint and assuming that $G$ deterministically satisfies the conclusions of Corollary~\ref{cor:largesets} and Lemma~\ref{lem:mediumsets}.
Since $p \gg (\log \log n)^{1/3} (\log n)^2 n^{-1/3}$, we can choose a function $\omega=o(\log n)$ tending to infinity with $n$ and such that $p \ge \omega (\log \log n)^{1/3} (\log n)^2 n^{-1/3}$. Note that our choice of $\omega$ satisfies the requirements of Corollary~\ref{cor:largesets}. 
Call a set $S \subseteq V(G)$ \textbf{large} if $|S|=s \geq n/(\omega \log^2 n)$, \textbf{small} if $s \leq np/(\omega \log^2 n)$, and \textbf{medium} otherwise.

\medskip

Whenever Mrs.\ Correct is presented a large set, she can, by Corollary~\ref{cor:largesets} find an independent set of size $k_0=(2+o(1)) \log_b (np)$, and uses erasers for all remaining vertices. Note that, trivially, at most 
\begin{equation}\label{eq:large}
\frac {n}{(2+o(1)) \log_b (np)} = (1+o(1)) \frac {n}{2 \log_b (np)}
\end{equation}
large sets can be presented to her before the end of the game, and hence, for all large sets at most that many erasers on each vertex are needed.

\medskip

Suppose now that a small set $S$ is presented to Mrs.\ Correct. Then, she chooses a random vertex $v \in S$ and accepts the colour on that vertex; for all other vertices of the presented set erasers are used. (Note that this is clearly not a optimal strategy; Mrs.\ Correct could extend $\{v\}$ to a maximal independent set but we do not use it in the argument and so we may pretend that a single vertex is accepted.)  Let $X_v$ denote the random variable counting the number of erasers used when small sets containing $v$ are presented before eventually $v$ gets a permanent colour (which does not necessarily happen when a small set is presented). We have
\begin{equation}\label{eq:smallset}
\pr \left( X_v \geq \frac {np}{\sqrt{\omega} \log n} \right) \leq \left(1- \frac{\omega \log^2 n}{np}\right)^{np/( \sqrt{\omega}\log n)} \leq \exp(-\sqrt{\omega}\log n) =o(n^{-1}),
\end{equation}
and thus, by a union bound, a.a.s.\ for all vertices the number of erasers used for small sets is at most 
\begin{equation}\label{eq:small}
\frac {np}{\sqrt{\omega}\log n} =o \left( \frac{n}{2 \log_b (np)} \right),
\end{equation}
and so is negligible.

\medskip

Now, we are going to deal with medium sets. First, note that  the size of a medium set is at least $np/(\omega \log^2 n) \ge 10 \log n / p$ for the range of $p$ considered in this theorem and by our choice of $\omega$. Suppose that some medium set $S$ is presented during the game. Applying Lemma~\ref{lem:mediumsets} repeatedly, Mrs.\ Correct can partition $S$ into independent sets of size $\lceil 1/(9p) \rceil$ and a remaining set $J$ of size at most $10 \log n / p \le np/(\omega \log^2 n)$. The strategy is then the following: with probability $1/2$ she chooses (uniformly at random) one of the independent sets of size $\lceil 1/(9p) \rceil$, and with probability $1/2$ she chooses one vertex chosen uniformly at random from $J$ (as before, this is clearly a suboptimal strategy but convenient to analyze). Selected vertices keep the colour; for the others erasers need to be used. We partition the vertices of $S$ into two groups: group $1$ consists of vertices belonging to independent sets, and group $2$ consists of vertices of $J$. Our goal is to show that each vertex $v$ appears in less than $3np/(\sqrt{\omega}\log n)$ many medium sets, so that the total number of erasers needed to deal with these situations is negligible (see~\eqref{eq:small}). 

Suppose that some vertex $v$ appears in at least  $3np/(\sqrt{\omega}\log n)$ medium sets. Each time, the corresponding medium set is split into two groups, and we cannot control in which group $v$ ends up. However, by Chernoff's bound, with probability $1-o(n^{-1})$, at least $np/(\sqrt{\omega}\log n)$ times the group to which $v$ belongs  is chosen. We condition on this event. Note that if group $2$ is chosen and $v$ belongs to this group, the probability that $v$ is not selected is at most $1- \omega \log^2 n/(np)$ (as in~\eqref{eq:smallset} for small sets). Similarly, if group $1$ is chosen and $v$ belongs to this group, the probability that $v$ is not chosen, is at most $1 - \omega \log^2 n/(9np),$ as each medium set has size at most $n / (\omega \log^2 n)$.

Denote by $Y_v$ the random variable counting the number of erasers used for vertex $v$ corresponding to medium sets, before eventually $v$ gets a permanent colour. As in~\eqref{eq:smallset}, 
\begin{eqnarray*}
\pr \left(Y_v \geq \frac {3np}{\sqrt{\omega} \log n} \right) &\leq& o(n^{-1}) + \left(1- \frac{\omega \log^2 n}{9np}\right)^{(np)/( \sqrt{\omega}\log n)} \\
&\leq& o(n^{-1}) + \exp \left(- \frac {\sqrt{\omega}\log n}{9} \right)=o(n^{-1}),
\end{eqnarray*}
and thus, as before, by a union bound, a.a.s.\ the desired bound for the number of erasers used for medium sets holds for all vertices. 

Combining bounds for the number of erasers used for large, medium, and small sets, we get that, regardless of the strategy used by Mr.\ Paint, Mrs.\ Correct uses at most $(1+o(1)) n/ (2 \log_b (np))$ erasers for each vertex, and the theorem follows.

\section{Proof of Theorem~\ref{thm:main2}}

As in the proof of Theorem~\ref{thm:main}, the lower bound follows immediately from~\eqref{eq:chi}, and so it remains to show that Mrs.\ Correct has a strategy that a.a.s.\ requires only $O(n / \log_b (np))$ erasers on each vertex (recall also~\eqref{eq:chi2}). We will use the same definitions for sets of being small, medium, and large as before, but we set here $\omega=\log \log n$; as pointed out right after Lemma~\ref{lem:largesets2}, other choices of $\omega$ are possible, but our choice gives the strongest result (as we assume the weakest condition). The argument for small sets remains exactly the same; a.a.s.\ it is enough to have $o(n / \log_b (np))$ erasers on each vertex to deal with all small sets that Mr.\ Paint may present. To deal with large sets, by Lemma~\ref{lem:largesets2}, since we aim for a  statement that holds a.a.s., we may assume that whenever a large set is presented, Mrs.\ Correct can always find an independent set of size at least $\eps(1-\eps) \log (np)/(3p)$, keeps the colour on these vertices, and uses erasers for all remaining vertices. The number of large sets presented is therefore at most $3 (\eps(1-\eps))^{-1}  np / \log(np) = O(n / \log_b (np))$, as needed. This is enough to determine the order of the on-line choice number but, if one aims for a better constant, then Remark~\ref{rem:densecase} implies that for $np = \log^{C+o(1)} n$ we are guaranteed to have an independent set of size $i=i(C)$, where
$$
i = (1+o(1))
\begin{cases}
\log (np)/p  & \text{if } C \to \infty \\
(1-\frac {2}{C}) \log (np)/p  & \text{if } C \in [4, \infty) \\
\frac {1}{2} \log (np)/p  & \text{if } C \in (2,4).
\end{cases}
$$
We will show below that the contribution of medium sets is of negligible order, and hence the upper bound on the number of erasers needed for large sets will also apply (up to lower order terms) to the total number of erasers needed. As a result, we will get the following bounds: $(2+o(1) \chi(G)$ for $C \to \infty$, $(\frac {2C}{C-2} + o(1)) \chi(G)$ for $C \in [4, \infty)$, and $(4+o(1)) \chi(G)$ for $C \in (2, 4)$.

The strategy for medium sets has to be substantially modified. Suppose that a medium set $S$ of size $s$ is presented at some point of the game. Recall that $np/(\omega \log^2 n) < s < n/(\omega \log^2 n)$, where $\omega=\log\log n$. Since we aim for a statement that holds a.a.s., we may assume that Mrs.\ Correct can partition $S$ in the following way: by applying Lemma~\ref{lem:mediumsets}(i) repeatedly, as long as at least $s_0 = 10 \log n/p$ vertices are remaining, she can find independent sets of size $\lceil 1/(9p) \rceil$, and remove them from $S$. (Note that for sparse graphs it might happen that $s < s_0$ and so the lemma cannot be applied even once. In such a situation, we simply move on to the next step.) If the number of remaining vertices, $r$, satisfies $2^{-i} s_0 \leq r < 2^{-i+1} s_0$ for some $i = i(n) \in \mathbb{N}$ and $r > np/(\omega \log^2 n)$, then by Lemma~\ref{lem:mediumsets}(ii), she can find an independent set of size $\lceil i/(9p2^{i}) \rceil$. Then, she removes that independent set, and continues iteratively with the remaining vertices in $S$. Note that there are clearly $M \le \log n / \log 2 =O(\log n)$ different sizes of independent sets corresponding to different values of $i$. Finally, if $r \leq np/(\omega \log^2 n)$, she puts all the remaining vertices into a set $J$ (not necessarily independent), and stops partitioning.  

We classify vertices in $S$ into types depending on the size of the set in the partition of $S$ to which they belong:
more precisely, we call vertices from independent sets of size $\lceil 1/(9p) \rceil$ to be of \textbf{type-$0$}; vertices from independent sets of size $\lceil i/(9p2^{i}) \rceil$ for some $1 \leq i \leq M$ are called of \textbf{type-$i$}; and vertices from the last set of vertices $J$ of size at most $np/(\omega \log^2 n)$ are called of \textbf{type-$(M+1)$}. Of course, these definitions depend on the particular time a set $S$ is presented by Mr.\ Paint. In particular,
if a vertex is presented several times by Mr.\ Paint, each time it might be of a different type. Finally, we classify vertices in $S$ into 3 groups: vertices of type-0 form \textbf{group-$0$}, vertices of type-$1$ up to type-$M$ form \textbf{group-$1$}, and vertices of type-$(M+1)$ form \textbf{group-$2$}.

Mrs.\ Correct now chooses with probability $1/3$ one of the three groups. If group-$0$ is selected, she then picks uniformly at random one independent set within the group. In case when group-$2$ is selected, she picks uniformly at random a single vertex inside this group. Finally, if group-$1$ is selected, she first picks a random type. The probability that type-$i$ is selected is equal to
$$
q_i := \frac {1/i}{\sum_{j=1}^{M} 1/j} \ge \frac {1/i}{\sum_{j=1}^{\log n / \log 2} 1/j} = \frac {1+o(1)}{i \log \log n}.
$$
Then, within the selected type, an independent set is selected uniformly at random. If the group or type chosen by Mrs.\ Correct has no vertices in it, she picks no vertex to colour permanently, which is clearly a suboptimal strategy.
Our goal is to show that a.a.s.\ each vertex $v$ appears in less than
$$
\frac {4np}{\sqrt{\omega}\log (np)} = o \left( \frac{n}{\log_b (np)} \right)=o(\chi(G))
$$ 
many medium sets, before its colour is accepted. If this holds, then the number of erasers needed for each vertex (due to medium sets) is negligible.

Suppose that some vertex $v$ appears in at least $4np/(\sqrt{\omega}\log (np))$ medium sets. Since one of the three groups is selected uniformly at random for each medium set, we expect $v$ to belong to the selected group at least $(4/3)np/(\sqrt{\omega}\log (np)) \ge \log^{2+\eps+o(1)} n$ times. Hence, it follows from Chernoff's bound that, with probability $1-o(n^{-1})$, at least $np/(\sqrt{\omega}\log (np))$ times the group to which $v$ belongs is chosen. Call this event $E$. It remains to show that, conditional on $E$, $v$ will be permanently coloured  with probability $1-o(n^{-1})$ within the first $np/(\sqrt{\omega}\log (np))$ times its group is picked for some medium set.

Note that if group-$0$ is selected and $v$ belongs to this group, the probability that $v$ is chosen to be permanently coloured, is at least
\begin{equation}\label{eq:g0}
 \frac{1/(9p)}{n/(\omega \log^2 n)}  =  \frac{\omega \log^2 n}{9np}.
\end{equation}
Suppose then that group-$1$ is chosen, $v$ belongs to this group and is of type-$i$ for some $1 \leq i \leq M$.
The number of independent sets in the partition containing vertices of type-$i$ is at most 
\[
\frac{2^{-i+1}s_0/2}{i/(9p2^i)} + 1 = \frac{9s_0p}{i} + 1 = \frac{90\log n}{i} + 1 \le \frac{100\log n}{i},
\]
where for the last step we used that $i\le M\le \log n/\log 2$.  This time, the probability that $v$ is permanently coloured is at least
\begin{equation}\label{eq:g1}
q_i \frac{i}{100\log n} \ge \frac{1+o(1)}{100(\log n)(\log\log n)}   \ge \frac{\omega \log(np) \log n}{np},
\end{equation}
where we used that $\omega=\log\log n$ and $np/(\log(np)) \ge \log^{2+\eps/2}n$. Finally, if group-$2$ is chosen and $v$ belongs to this group, the probability that $v$ is permanently coloured is at least
\begin{equation}\label{eq:g2}
\frac{1}{np / (\omega \log^2 n)} = \frac{\omega \log^2 n}{np}.
\end{equation}
Summarizing~\eqref{eq:g0}, \eqref{eq:g1} and~\eqref{eq:g2}, whenever the group of $v$ is selected (regardless of which group it is), the probability that $v$ is chosen to be permanently coloured is at least
$$
\frac{\omega \log(np) \log n}{9np}.
$$

The rest of the argument works as before. Given a vertex $v$ which is presented in at least $4np/(\sqrt{\omega}\log (np))$ medium sets, denote by $Y_v$ the random variable counting the number of erasers used by $v$ (due to medium sets) before $v$ gets a permanent colour. Then,
 \begin{eqnarray*}
\pr \left( Y_v \geq \frac { 4 np}{\sqrt{\omega} \log (np)} \; \Big| \;  E \right) &\leq&  \left(1-     \frac{\omega \log(np) \log n}{9np}        \right)^{np/( \sqrt{\omega}\log (np))} \\
&\le&  \exp \left(- \Omega \left( 
  \sqrt\omega  \log n  \right) \right)
= o(n^{-1}).
\end{eqnarray*}
Since $\pr(E)=1-o(n^{-1})$,
the unconditional probability that $Y_v \ge 4 np/(\sqrt{\omega} \log (np))$ is also $o(n^{-1})$.  Hence, regardless of the strategy followed by Mr.\ Paint, we can take a union bound over all vertices that were presented in at least $4 np/(\sqrt{\omega} \log (np))$ medium sets, and deduce that a.a.s.\ for every vertex the number of erasers used (due to medium sets) is less than $4 np/(\sqrt{\omega} \log (np))$ and thus negligible.
The proof of the theorem is finished.

\section{Proof of Theorem~\ref{thm:main3}}

Before we move to the proof of Theorem~\ref{thm:main3}, let us state the following simple observation.

\begin{lem}\label{lem:trees}
Let $G$ be a graph whose components are all trees or unicyclic graphs. Then Mrs.\ Correct has a winning strategy using $1$ eraser at each vertex for tree components and using $2$ erasers at each vertex for unicyclic components. 
\end{lem}

\begin{proof}
Since we may consider different components separately, we may assume that $G$ is connected. First assume that $G$ is a tree. For a contradiction, suppose that $1$ eraser at each vertex is not enough for Mrs.\ Correct to win on some tree, and consider a smallest such tree $T$. Clearly, $|T| \geq 2$, and let us consider a leaf $\ell$ of $T$. By minimality of $T$, Mrs.\ Correct has a winning strategy on $T \setminus \ell$. Then, she extends the strategy to $T$ as follows. The first time she is presented vertex $\ell$, she considers the optimal strategy when playing on the restriction of the set to $T \setminus \ell$. If the set yielded by that strategy does not contain the neighbour of $\ell$ in $T$, she follows this strategy and simply adds $\ell$ to the set, and plays for the rest of the game on $T \setminus \ell$. On the other hand, if the set does contain the neighbour of $\ell$, then she also follows the strategy on $T \setminus \ell$, but uses the eraser for $\ell$. Since the only neighbour of $\ell$ cannot appear later on, she can continue with her optimal strategy on $T \setminus \ell$ and simply adds $\ell$ the second time she is presented this vertex. We get a contradiction, and the proof of the first part is finished.

Similarly, suppose now that $G$ is unicyclic. As before, for a contradiction, suppose that $2$ erasers at each vertex are not enough for $G$, and consider a smallest unicyclic graph $U$ with this property. Clearly, $|U| \geq 3$. If $U$ contains a leaf $\ell$, then, as before, she plays optimally using two erasers on $U \setminus \ell$ and adapts her strategy exactly as before. If $U$ does not contain a leaf, then $U$ is a cycle. The first time she is presented a set, she picks one vertex and uses erasers for all other vertices. The rest of the game is played on a tree (in fact, a path), and she has still $1$ eraser at each vertex at her disposal. By the first part of the lemma, she has a winning strategy and the proof is finished.  
\end{proof}

We come now back to the proof of Theorem~\ref{thm:main3}. It is well known that if $p < 0.99/n$, then a.a.s.\ $G$ contains only trees and unicyclic components, and Mrs.\ Correct can win using at most $2$ erasers by Lemma~\ref{lem:trees}.  Assume then that $p\le c/n$ for some (perhaps large) constant $c \ge 0.99$ (assume w.l.o.g.\ that $c\in\nat$). If $G$ contains a component with at least two cycles, the component must contain a subgraph which either consists of two  cycles connected by a path (or sharing a vertex), or is a cycle with a ``diagonal'' path. Let us call such structures \textbf{complex}. Clearly, on $k$ vertices one can construct at most $k^2 k!$ complex structures. Note that the degree of any vertex given by the existence of the complex structure is at most $4$. We will show that a.a.s.\ there is no complex structure with the property that each vertex of this structure has a degree in the whole graph of  at least $100c^2$. Note that, given a complex structure, in order for each vertex of this structure to have degree in the whole graph at least $100c^2$, it must have either at least $47c^2$ incident edges towards vertices outside the subgraph on which the complex structure is built, or it has at least $47c^2$ additional incident edges inside the subgraph on which the complex structure is built. Therefore, either half of the vertices of the complex structure have at least $47c^2$ incident edges outside, or there are at least $11c^2k$ additional edges inside the complex structure, and no information about these edges has been revealed so far.  Thus, the expected number of complex structures in which each vertex has degree in the whole graph at least $100c^2$ in the whole graph is at most 
\begin{align}
\sum_{k=4}^n {n \choose k} k^2 k! p^{k+1} &\left( \binom{k}{\lfloor k/2 \rfloor} \left(  {n \choose 47c^2} p^{47c^2}\right)^{\lfloor k/2 \rfloor} + \binom{\binom{k}{2}} {11c^2k}p^{11c^2k}\right) \nonumber \\
& \leq \sum_{k=4}^n \frac{k^2}{n}c^{k+1} 2\left(\frac{e}{22 c}\right)^{11c^2k} = O\left( \sum_{k=4}^n \frac {k^2}{n} c^{k+1} c^{-2k} \right) \nonumber \\
&= O\left( \sum_{k=4}^n \frac {k^2}{n} c^{-k} \right) = O\left( \int_{x=0}^{\infty} \frac {x^2}{n}c^{-x} dx \right) = O(1/n),
\label{eq:densegraph}
\end{align}
and therefore, by Markov's inequality, the subgraph induced by vertices of degree at least $100c^2$ a.a.s.\ consists of components that are either trees or unicyclic components. Since we aim for a statement that holds a.a.s., we may assume that this is the case.

We may assume that each time Mrs.\ Correct selects a maximal independent set. Hence, the number of erasers needed to be placed at vertex $v$ is at most $\deg(v)$ (if $v$ uses one of its erasers at some point of the game, it must be the case that at least one of its neighbours belongs to the maximal independent set). As a result, vertices of degree at most $100c^2$ require only a constant number of erasers. Call the set of such vertices $L$, and let $H=V(G)\setminus L$.  By~\eqref{eq:densegraph}, the graph induced by the vertices in $H$ a.a.s.\ consists of components that are either trees or unicyclic components. Mrs.\ Correct plays as follows: whenever she is presented a set $U$, she plays optimally on the restriction of $U$ to $H$ (on which, by Lemma~\ref{lem:trees} she uses at most 2 erasers), and then she extends an independent set found there to any maximal independent set in $U$. This strategy uses in total at most $100c^2=O(1)$ erasers for each vertex in $L$. In this way, clearly $O(1)=\Theta(\chi(G))$ colours are used.

\end{document}